\documentclass{amsart}
\usepackage{amsmath, amstext, amsbsy, amssymb}
\usepackage{graphicx}
\usepackage{subfigure}
\usepackage{picinpar}
\usepackage{psfrag}

\hoffset \voffset \oddsidemargin=55pt \evensidemargin=55pt
\numberwithin{equation}{section}
\def\beq{\begin{eqnarray}}
\def\eeq{\end{eqnarray}}
\def\beqs{\begin{eqnarray*}}
\def\eeqs{\end{eqnarray*}}

\def\mz{{\mathbb Z}}
\def\mc{{\mathbb C}}
\input xy \xyoption{all}

\def\dim{{\hbox{\rm dim}}}

\def\ad{{\hbox{\rm ad}}}

\newfont{\df}{eufm10}

\def\gim{\hbox{\bf gim}}

\def\sg{\textsf{g}}
\def\sh{\textsf{h}}

\def\cu{\mathcal{U}}
\def\cl{\mathcal{L}}
\def\ca{\mathcal{A}}

\def\mc{{\mathbb C}}

\def\deg{\hbox{\rm deg}}

\title[Representations for gim algebras]
{Finite-dimensional representations for a class of generalized intersection matrix Lie algebras}
\thanks{$^\dag$the corresponding author's email: xialimeng@ujs.edu.cn}

\author[Y. Gao]{Yun Gao$^1$}
\author[L. Xia]{Li-meng   Xia$^{2,\dag}$}

\date{}
\begin{document}
\maketitle

\centerline{$^1$Department of Mathematics and Statistics, York University,  Canada}

\centerline{$^2$Faculty of Science, Jiangsu University, P.R. China}

\def\abstractname{ABSTRACT}
\begin{abstract}
In this paper, we study a class of generalized intersection matrix Lie algebras $\gim(M_n)$, and prove that  its every finite-dimensional semi-simple quotient is  of type $M(n,{\bf a}, {\bf c},{\bf d})$. Particularly, any finite dimensional irreducible $\gim(M_n)$ module must be an irreducible module of $M(n,{\bf a}, {\bf c},{\bf d})$ and any finite dimensional irreducible $M(n,{\bf a}, {\bf c},{\bf d})$ module must be an irreducible module of $\gim(M_n)$.

\vskip3mm \noindent {\it Key Words}: intersection matrix algebras; irreducible modules; quotient algebras; affine Lie algebras.
\end{abstract}

\newtheorem{theo}{Theorem}[section]
\newtheorem{theorem}[theo]{Theorem}
\newtheorem{defi}[theo]{Definition}
\newtheorem{lemma}[theo]{Lemma}
\newtheorem{coro}[theo]{Corollary}
\newtheorem{proposition}[theo]{Proposition}
\newtheorem{remark}[theo]{Remark}

\setcounter{section}{0}

\section{Introduction}

In the early to mid-1980s, Peter Slodowy discovered that matrices like
$$M =\left[\begin{array}{cccc}
2&-1&0&1\\
-1&2&-1&1\\
0& -2&2&-2\\
1&1&-1&2
\end{array}\right]$$
were encoding the intersection form on the second homology group of
Milnor fibres for germs of holomorphic maps with an isolated singularity
at the origin [S1], [S2]. These matrices were like the generalized
Cartan matrices of Kac-Moody theory in that they had integer entries,
$2$'s along the diagonal, and $m_{i,j}$ was negative if and only if $m_{j,i}$ was
negative. What was new, however, was the presence of positive entries
off the diagonal. Slodowy called such matrices generalized intersection
matrices:

\begin{defi}{\rm (\cite{S1})} An $n\times n$ integer-valued matrix $M=(m_{i,j})_{1\leq i,j\leq n}$ is called a
generalized intersection matrix ({\gim}) if

\qquad $m_{i,i}=2$,

\qquad $m_{i,j}<0$ if and only if $m_{j,i}<0$, and

\qquad $m_{i,j}>0$ if and only if $m_{j,i}>0$

{\noindent for $1\leq i, j\leq n$ with $i\not= j$.}
\end{defi}

Slodowy used these matrices to define a class of Lie algebras that
encompassed all the Kac-Moody Lie algebras:
\begin{defi}[see \cite{BrM}, \cite{S1}] Given an $n\times n$ generalized intersection
matrix $M = (m_{i,j})$, define a Lie algebra over $\mc$, called a generalized
intersection matrix (GIM) algebra and denoted by $\gim(M)$, with:

generators: $e_1, . . . , e_n, f_1, . . . , f_n, h_1, . . . h_n$,

relations:
(R1) for $1\leq i, j \leq n$,
\beqs [h_i, e_j ] = m_{i,j}e_j,&{[h_i, f_j ]} =-m_{i,j}f_j,&{[e_i, f_i]} = h_i,\eeqs
\qquad (R2) for $m_{i,j}\leq 0$,
\beqs [e_i, f_j ] = 0 = [f_i, e_j ],&&(\ad e_i)^{-m_{i,j}+1}e_j = 0 = (\ad f_i)^{-m_{i,j}+1}f_j,\eeqs
\qquad (R3) for $m_{i,j} > 0$, $i\not= j$,
\beqs [e_i, e_j ] = 0 = [f_i, f_j ],&&(\ad e_i)^{m_{i,j}+1}f_j = 0 = (\ad f_i)^{m_{i,j}+1}e_j.\eeqs
\end{defi}

If the $M$ that we begin with is a generalized Cartan matrix, then
the $3n$ generators and the first two groups of axioms, (R1) and (R2),
provide a presentation of the Kac-Moody Lie algebras [GbK], [C], [K].

Slodowy and, later, Berman showed that the GIM algebras are also
isomorphic to fixed point subalgebras of involutions on larger Kac-
Moody algebras [S1], [Br]. So, in their words, the GIM algebras lie
both "beyond and inside" Kac-Moody algebras.

Further progress came in the 1990s as a byproduct of the work of
Berman and Moody, Benkart and Zelmanov, and Neher on the classification
of root-graded Lie algebras [BrM], [BnZ], [N]. Their work
revealed that some families of intersection matrix ($\emph{\textbf{im}}$) algebras, were universal covering algebras
of well understood Lie algebras.
An $\emph{\textbf{im}}$ algebra generally is a quotient algebra of a GIM algebra associated to the ideal generated  by homogeneous vectors those have long roots (i.e., $(\alpha,\alpha)>2$).

A handful of other researchers also began
engaging these new algebras. For example, Eswara Rao, Moody, and
Yokonuma used vertex operator representations to show that $\emph{\textbf{im}}$ algebras
were nontrivial [EMY]. Analogous compact forms of $\emph{\textbf{im}}$ algebras
arising from conjugations over the complex field were considered in [G]. Peng found
relations between $\emph{\textbf{im}}$ algebras and the representations of tilted algebras
via Ringel-Hall algebras [P]. Berman, Jurisich, and Tan showed
that the presentation of GIM algebras could be put into a broader
framework that incorporated Borcherds algebras [BrJT].

In present paper, we study the GIM algebra $\gim(M_n)$ constructed through intersection matrix
\beqs M_n=(m_{i,j})_{n\times n}=\left[\begin{array}{cccccc}
2&-1&0&\cdots&0&1\\
-1&2&-1&\cdots&0&0\\
0&-1&2&\cdots&0&0\\
\vdots&\vdots&\vdots&\vdots&\vdots&\vdots\\
0&0&0&\cdots&2&-1\\
1&0&0&\cdots&-1&2
\end{array}\right]_{n\times n}\eeqs
where $n\geq 3$, and we build the representation theory of finite-dimensional modules for $\gim(M_n)$.

Generally, $\gim(M_n)$ can be illustrated by the following diagram (also called Dynkin diagram):

\centerline{\xymatrix{
^1\circ\ar@{.}[rrd]\ar@{-}[r]&^2\circ\ar@{-}[r]&^3\circ-&-^{n-2}\circ\ar@{-}[r]&\circ^{n-1}\ar@{-}[lld]\\
&&_n\circ &&&}}
{\noindent}where the numbered circles present the indices of $e_i, f_i$ ($1\leq i\leq n$), the solid line between two circles $i,i+1$ means that $m_{i,i+1}=m_{i+1,i}=-1$ and the unique dotted line means that $m_{1,n}=m_{n,1}=1$. There is no line between any  other pair $(i,j)$, it means that $m_{i,j}=m_{j,i}=0$.

\section{Construction of epimorphisms}

\begin{defi}
Suppose that $\cl$ is a non-trivial semi-simple Lie algebra and ${\bf a}\in\mz_{\geq0}, {\bf c, d}\in\{0,1\}$. If $n\geq 3$ and $\cl$ is isomorphic to a direct sum of ${\bf a}$ copies of $sl_{2n}$, ${\bf c}$ copies of $sp_{2n}$ and ${\bf d}$ copies of $so_{2n}$, then  we say that $\cl$ is of $M(n,{\bf a}, {\bf c}, {\bf d})$ type.
\end{defi}

Let $\cl=\oplus_{k=1}^K\cl_k$ be of type $M(n,{\bf a}, {\bf c}, {\bf d})$, in  this section we  construct  an epimorphism from $\gim(M_n)$ to $\cl$.

For convenience, we fix a Chevalley generators $\{e_{\alpha_i}, f_{\alpha_i}|1\leq i\leq n\}$ for simple Lie algebra of type $C_n$ or $D_n$, and $\{e_{\alpha_i}, f_{\alpha_i}|1\leq i\leq 2n-1\}$ for $A_{2n-1}$. The associative Dynkin diagrams are

\centerline{\xymatrix{D_{n}:&^1\circ\ar@{-}[r]&^2\circ\ar@{-}[r]\ar@{-}[d]&^3\circ-&-\circ\ar@{-}[r]&\circ^{n-1}\\
&&_n\circ}}
\centerline{\xymatrix{C_{n}:&^n\circ\ar@{=>}[r]&^1\circ\ar@{-}[r]&^2\circ-&-\circ\ar@{-}[r]&\circ^{n-1}}}
\centerline{\xymatrix{A_{2n-1}:&^1\circ\ar@{-}[r]&^2\circ\ar@{-}[r]&^3\circ-&-\circ\ar@{-}[r]&\circ^{2n-1}}}

\begin{lemma}\label{L2.1}
Let $a\in\mc^\times, a\not=\pm1$. If $\{e_{\alpha_i},f_{\alpha_i}|1\leq i\leq 2n-1\}$ is the Chevalley generators of $A_{2n-1}$, let
\beqs e_{\alpha_{2n}}&=&a[f_{\alpha_{2n-1}},\cdots[f_{\alpha_2},f_{\alpha_1}]\cdots],\\
f_{\alpha_{2n}}&=&a^{-1}[\cdots[e_{\alpha_1},e_{\alpha_2}]\cdots, e_{\alpha_{2n-1}}],\eeqs
then \beqs e_i\mapsto e_{\alpha_i}-f_{\alpha_{n+i}},&&f_i\mapsto f_{\alpha_i}-e_{\alpha_{n+i}},\eeqs
defines a Lie algebra homomorphism from $\gim(M_n)$ to $A_{2n-1}$, where $1\leq i\leq n$.\end{lemma}
\begin{proof}
Let \beqs X_i&=&e_{\alpha_i}-f_{\alpha_{n+i}},\\
Y_i&=&f_{\alpha_i}-e_{\alpha_{n+i}},\\
H_i&=&[X_i,Y_i],\eeqs
for all $1\leq i\leq n$.

Note that the subalgebra generated by $e_{\alpha_i},f_{\alpha_i} (1\leq i\leq 2n-1, i\not=n)$ is isomorphic to $sl_n\oplus sl_n$. So the map
 \beqs e_i\mapsto e_{\alpha_i}-f_{\alpha_{n+i}},&&f_i\mapsto f_{\alpha_i}-e_{\alpha_{n+i}},\eeqs
restricted to $1\leq i\leq n-1$ induces  a diagonal injective map $\Phi_1\oplus\Phi_2$, where
 \beqs \Phi_1: &e_i\mapsto e_{\alpha_i},&f_i\mapsto f_{\alpha_i},\\
 \Phi_2:& e_i\mapsto -f_{\alpha_{n+i}},&f_i\mapsto -e_{\alpha_{n+i}}.\eeqs
Particularly, $\Phi_1$ can be viewed as an inclusion and $\Phi_2$ can be viewed the composition of Chevalley involution and inclusion.

It is sufficient to check the relations involving elements $X_n,Y_n,H_n$.
\beqs H_i&=&[e_{\alpha_i}-f_{\alpha_{n+i}}, f_{\alpha_i}-e_{\alpha_{n+i}}]=h_{\alpha_i}-h_{\alpha_{n+i}},\;\forall\; 1\leq i\leq n-1,\\
H_n&=&[e_{\alpha_n}-f_{\alpha_{2n}}, f_{\alpha_n}-e_{\alpha_{2n}}]=h_{\alpha_n}+(h_{\alpha_1}+\cdots+h_{\alpha_{2n-1}}),\eeqs
where $h_{\alpha_j}=[e_{\alpha_j}, f_{\alpha_j}]$ for all $1\leq j\leq 2n-1$. ({\it For the computation of $H_n$, see Remark 1.}) Then
\beqs [H_n, X_n]&=&[h_{\alpha_n}+(h_{\alpha_1}+\cdots+h_{\alpha_{2n-1}}),e_{\alpha_n}-f_{\alpha_{2n}}]\\
&=&\alpha_n(h_{\alpha_{n-1}}+2h_{\alpha_n}+h_{\alpha_{n+1}})e_{\alpha_n}-(\alpha_1+\cdots+\alpha_{2n-1})(h_{\alpha_n}+(h_{\alpha_1}+\cdots+h_{\alpha_{2n-1}}))f_{\alpha_{2n}}\\
&=&\alpha_n(h_{\alpha_n})e_{\alpha_n}-(\alpha_1+\cdots+\alpha_{2n-1})(h_{\alpha_1}+h_{\alpha_{2n-1}})f_{\alpha_{2n}}\\
&=&2X_n,\\
{[H_n, Y_n]}&=&-\alpha_n(h_{\alpha_n})f_{\alpha_n}+(\alpha_1+\cdots+\alpha_{2n-1})(h_{\alpha_1}+h_{\alpha_{2n-1}})e_{\alpha_{2n}}\\
&=&-2Y_n,
\eeqs
and
\beqs [H_n, X_1]&=&\alpha_1(h_{\alpha_n}+(h_{\alpha_1}+\cdots+h_{\alpha_{2n-1}}))e_{\alpha_1}-(-\alpha_{n+1})(h_{\alpha_n}+(h_{\alpha_1}+\cdots+h_{\alpha_{2n-1}}))f_{\alpha_{n+1}}\\
&=&\alpha_1(h_{\alpha_1}+h_{\alpha_2})e_{\alpha_1}+\alpha_{n+1}(2h_{\alpha_n}+h_{\alpha_{n+1}}+h_{\alpha_{n+2}})f_{\alpha_{n+1}}\\
&=&X_1,\\
{[H_n, Y_1]}&=&-\alpha_1(h_{\alpha_1}+h_{\alpha_2})f_{\alpha_1}-\alpha_{n+1}(2h_{\alpha_n}+h_{\alpha_{n+1}}+h_{\alpha_{n+2}})e_{\alpha_{n+1}}\\
&=&-Y_1,\\
{[H_n, X_{n-1}]}&=&\alpha_{n-1}(h_{\alpha_n}+(h_{\alpha_1}+\cdots+h_{\alpha_{2n-1}}))e_{\alpha_{n-1}}-(-\alpha_{2n-1})(h_{\alpha_n}+(h_{\alpha_1}+\cdots+h_{\alpha_{2n-1}}))f_{\alpha_{2n-1}}\\
&=&\alpha_{n-1}(h_{\alpha_{n-2}}+h_{\alpha_{n-1}}+2h_{\alpha_n})e_{\alpha_{n-1}}+\alpha_{2n-1}(h_{\alpha_{2n-2}}+h_{\alpha_{2n-1}})f_{\alpha_{2n-1}}\\
&=&-X_{n-1},\\
{[H_n, Y_{n-1}]}&=&-\alpha_{n-1}(h_{\alpha_{n-2}}+h_{\alpha_{n-1}}+2h_{\alpha_n})f_{\alpha_{n-1}}-\alpha_{2n-1}(h_{\alpha_{2n-2}}+h_{\alpha_{2n-1}})e_{\alpha_{2n-1}}\\
&=&Y_{n-1}.\eeqs
Similar argument implies that
\beqs &&[H_1, X_n]=X_n,[H_1, Y_n]=-Y_n,[H_{n-1}, X_n]=-X_n,[H_{n-1}, Y_n]=Y_n,\\
&&{[X_1, X_n]}=[Y_1, Y_n]=[Y_{n-1},X_n]=[X_{n-1},Y_n]=0,\\
&&[X_i,[X_i,Y_j]]=[Y_i,[Y_i,X_j]]=0,\hbox{\rm for\;}\{i,j\}=\{1,n\},\\
&&[X_i,[X_i,X_j]]=[Y_i,[Y_i,Y_j]]=0,\hbox{\rm for\;}\{i,j\}=\{n-1,n\}.
\eeqs
Finally, the following relation is clear:
\beqs
[X_n, X_i]=[X_n,Y_i]=[Y_n, X_i]=[Y_n,Y_i]=0\eeqs
for all $2\leq i\leq n-2$.
\end{proof}

{\bf Remark 1.}{\it \quad  Suppose that $\alpha, \beta, \alpha+\beta\in\Delta^+$ in $A_{2n-1}$ and $[e_\alpha, f_\alpha]=\alpha^\vee, [e_\beta, f_\beta]=\beta^\vee$.
Then it holds that $[[e_\alpha,e_\beta],[f_\beta, f_\alpha]]=\alpha^\vee+\beta^\vee$.
    Repeatedly using  this formula, we infer that
\beqs [f_{\alpha_{2n}}, e_{\alpha_{2n}}]=h_{\alpha_1}+\cdots+h_{\alpha_{2n-1}}.\eeqs
Together with that $\alpha_1+\cdots+\alpha_{2n-1}\pm\alpha_n$ is not a root, we obtain the computation of $H_n$.

One can also  understand it in an easy way: Let $A_{2n-1}$ be the matrix Lie algebra $sl_{2n}$ and $e_{\alpha_i}=E_{i,i+1},  f_{\alpha_i}=E_{i,i+1}$ for $i<2n$. Then $e_{\alpha_{2n}}=aE_{2n,1}, f_{\alpha_{2n}}=a^{-1}E_{1,2n}$, which implies that
\beqs [f_{\alpha_{2n}}, e_{\alpha_{2n}}]=E_{1,1}-E_{2n,2n}=\sum_{i=1}^{2n-1}(E_{i,i}-E_{i+1,i+1})=\sum_{i=1}^{2n-1}h_{\alpha_i}.\eeqs}

\begin{lemma}\label{L2.2}
If $\{e_{\alpha_i},f_{\alpha_i}|1\leq i\leq n\}$ is the Chevalley generators of $C_n$,
let \beqs E&=&[f_{\alpha_{n-1}},\cdots[f_{\alpha_1},f_{\alpha_n}]\cdots],\\
F&=&[\cdots[e_{\alpha_{n}},e_{\alpha_1}]\cdots, e_{\alpha_{n-1}}],\eeqs
then \beqs e_i\mapsto e_{\alpha_i},&&f_i\mapsto f_{\alpha_i},\qquad 1\leq i\leq n-1,\\
e_n\mapsto E,&&f_n\mapsto F,\eeqs
defines a Lie algebra homomorphism from $\gim(M_n)$ to $C_{n}$.\end{lemma}
\begin{proof}
It is sufficient to check the relation involving elements $E,F$.  Actually, $F$ (respectively $E$) is a root vector of highest short root (respectively lowest short root). Then
\beqs [F, e_{\alpha_i}]=[E, f_{\alpha_i}]=0&&\forall\; 2\leq i\leq n-1,\\
{[F, f_{\alpha_i}]}={[E, e_{\alpha_i}]}=0&&\forall\; 1\leq i\leq n-2,\eeqs
and
\beqs [F,[F,e_{\alpha_1}]]=[e_{\alpha_1},[e_{\alpha_1},F]]=0,&&[E,[E,f_{\alpha_1}]]=[f_{\alpha_1},[f_{\alpha_1},E]]=0,\\
  {[F,[F,f_{\alpha_{n-1}}]]}=[f_{\alpha_{n-1}},[f_{\alpha_{n-1}},F]]=0,&&[E,[E,e_{\alpha_{n-1}}]]=[e_{\alpha_{n-1}},[e_{\alpha_{n-1}},E]]=0.  \eeqs
Moreover, $H:=[E,F]=-(2h_{\alpha_n}+h_{\alpha_1}+\cdots+h_{\alpha_{n-1}})$, where $h_{\alpha_i}=[e_{\alpha_i},f_{\alpha_i}]$ for all $1\leq i\leq n$. Hence, we have
\beqs [H,E]&=&(-\alpha_1-\cdots-\alpha_n)(-(2h_{\alpha_n}+h_{\alpha_1}+\cdots+h_{\alpha_{n-1}}))E=2E,\\
{[H,F]}&=&(\alpha_1+\cdots+\alpha_n)(-(2h_{\alpha_n}+h_{\alpha_1}+\cdots+h_{\alpha_{n-1}}))F=-2F,\eeqs
and
\beqs [H,e_{\alpha_1}]=e_{\alpha_1},&&[H,f_{\alpha_1}]=-f_{\alpha_1},\\
{[H,e_{\alpha_{n-1}}]}=-e_{\alpha_{n-1}},&&[H,f_{\alpha_{n-1}}]=f_{\alpha_{n-1}},\\
{[h_{\alpha_1}},E]=E,&& {[h_{\alpha_{n-1}}},E]=-E,\\
{[h_{\alpha_1}},F]=-F,&& {[h_{\alpha_{n-1}}},F]=F.\eeqs
The above calculation implies our statement and the proof is completed.
\end{proof}
\begin{lemma}\label{L2.3}
If $\{e_{\alpha_i},f_{\alpha_i}|1\leq i\leq n\}$ is the Chevalley generators of $D_n$, let \beqs E&=&[f_{\alpha_{n-1}},\cdots[f_{\alpha_2},f_{\alpha_n}]\cdots],\\
F&=&[\cdots[e_{\alpha_{n}},e_{\alpha_2}]\cdots, e_{\alpha_{n-1}}],\eeqs
then \beqs e_i\mapsto e_{\alpha_i},&&f_i\mapsto f_{\alpha_i},\qquad 1\leq i\leq n-1,\\
e_n\mapsto E,&&f_n\mapsto F,\eeqs
defines a Lie algebra homomorphism from $\gim(M_n)$ to $D_{n}$.\end{lemma}
\begin{proof}
It is sufficient to check the relation involving elements $E,F$.  Actually, $F$ (respectively $E$) is a root vector of highest  root (respectively lowest root) of subalgebra generated by $e_{\alpha_i},f_{\alpha_i} (2\leq i\leq n)$. Then
\beqs [F, e_{\alpha_i}]=[E, f_{\alpha_i}]=0&&\forall\; 2\leq i\leq n-1,\\
{[F, f_{\alpha_i}]}={[E, e_{\alpha_i}]}=0&&\forall\; 2\leq i\leq n-2,\eeqs
and
\beqs {[F,[F,f_{\alpha_{n-1}}]]}=[f_{\alpha_{n-1}},[f_{\alpha_{n-1}},F]]=0,&&[E,[E,e_{\alpha_{n-1}}]]=[e_{\alpha_{n-1}},[e_{\alpha_{n-1}},E]]=0.  \eeqs
Notice that $\alpha_1+\cdots+\alpha_n$ is a root, and neither of $-\alpha_1+\alpha_2+\cdots+\alpha_n$, $2\alpha_1+\cdots+\alpha_n$ and $\alpha_1+2(\alpha_2+\cdots+\alpha_n)$ is a root, hence we have
\beqs [F,[F,e_{\alpha_1}]]=[e_{\alpha_1},[e_{\alpha_1},F]]=0,&&[E,[E,f_{\alpha_1}]]=[f_{\alpha_1},[f_{\alpha_1},E]]=0,\\
{[F, f_{\alpha_1}]}=0,&&{[E, e_{\alpha_1}]}=0.\eeqs

Moreover, $H:=[E,F]=-(h_{\alpha_2}+\cdots+h_{\alpha_{n}})$, where $h_{\alpha_i}=[e_{\alpha_i},f_{\alpha_i}]$ for all $1\leq i\leq n$. Hence, we have
\beqs [H,E]&=&(-\alpha_2-\cdots-\alpha_n)(-(h_{\alpha_2}+\cdots+h_{\alpha_{n}}))E=2E,\\
{[H,F]}&=&(\alpha_2+\cdots+\alpha_n)(-(h_{\alpha_2}+\cdots+h_{\alpha_{n}}))F=-2F,\eeqs
and
\beqs [H,e_{\alpha_1}]=e_{\alpha_1},&&[H,f_{\alpha_1}]=-f_{\alpha_1},\\
{[H,e_{\alpha_{n-1}}]}=-e_{\alpha_{n-1}},&&[H,f_{\alpha_{n-1}}]=f_{\alpha_{n-1}},\\
{[h_{\alpha_1}},E]=E,&& {[h_{\alpha_{n-1}}},E]=-E,\\
{[h_{\alpha_1}},F]=-F,&& {[h_{\alpha_{n-1}}},F]=F.\eeqs
The above calculation implies our statement and the proof is completed.
\end{proof}

Next we begin to explicitly construct an epimorphism from $\gim(M_n)$ to $\cl$.  The construction is divided into four distinguish cases.

{\noindent\bf Case 1.} $\cl_k\cong sl_{2n}$ for all $1\leq k\leq K$.

Let $\{e_{\alpha_i}^{[k]}, f_{\alpha_i}^{[k]}|1\leq i\leq 2n-1\}$ be the analogue of Chevalley generators of $A_{2n-1}$ in $\cl_k$ for all $k$. Choose  a $K$-tuple $\underline{a}:=(a_1,\cdots, a_K)\in(\mc^\times)^K$ such that $a_k\not=\pm1$ and $a_k\not=a_{j}^{\pm1}$ for all $k\not=j$. Set
\beqs e_{\alpha_{2n}}^{[k]}&=&a_k[f_{\alpha_{2n-1}}^{[k]},\cdots[f_{\alpha_2}^{[k]},f_{\alpha_1}^{[k]}]\cdots],\\
f_{\alpha_{2n}}^{[k]}&=&a_k^{-1}[\cdots[e_{\alpha_1}^{[k]},e_{\alpha_2}^{[k]}]\cdots, e_{\alpha_{2n-1}}^{[k]}],\eeqs
for all $1\leq k\leq K$, where $e_{\alpha_{2n}}^{[k]}$ and $f_{\alpha_{2n}}^{[k]}$ are root vectors of lowest root and highest root, respectively.

{\noindent\bf Case 2.} $\cl_1\cong sp_{2n}$ and $\cl_k\cong sl_{2n}$ for all $2\leq k\leq K$.

Let $\{e_{\alpha_i}^{[1]}, f_{\alpha_i}^{[1]}|1\leq i\leq n\}$ be the analogue of  Chevalley generators of $C_n$ in $\cl_1$. and  $\{e_{\alpha_i}^{[k]}, f_{\alpha_i}^{[k]}|1\leq i\leq 2n-1\}$ be the analogue of Chevalley generators of $A_{2n-1}$ in $\cl_k$ for $k\geq 2$. Choose  a $K$-tuple ${\underline{a}}:=(1,a_2,\cdots, a_K)\in(\mc^\times)^K$ such that $a_k\not=\pm1$ and $a_k\not=a_{j}^{\pm1}$ for all $k\not=j$. Set
\beqs e_{\alpha_{2n}}^{[k]}&=&a_k[f_{\alpha_{2n-1}}^{[k]},\cdots[f_{\alpha_2}^{[k]},f_{\alpha_1}^{[k]}]\cdots],\\
f_{\alpha_{2n}}^{[k]}&=&a_k^{-1}[\cdots[e_{\alpha_1}^{[k]},e_{\alpha_2}^{[k]}]\cdots, e_{\alpha_{2n-1}}^{[k]}],\eeqs
for all $2\leq k\leq K$, and
\beqs
e_{\alpha_{2n}}^{[1]}&=&f_{\alpha_n}^{[1]}-[\cdots[[e_{\alpha_{n}}^{[1]},e_{\alpha_1}^{[1]}],e_{\alpha_2}^{[1]}]\cdots, e_{\alpha_{n-1}}^{[1]}],\\
f_{\alpha_{2n}}^{[1]}&=&e_{\alpha_n}^{[1]}-[f_{\alpha_{n-1}}^{[1]},\cdots[f_{\alpha_2}^{[1]},[f_{\alpha_1}^{[1]},f_{\alpha_n}^{[1]}]]\cdots],\\
&&e_{\alpha_{n+1}}^{[1]}=\cdots=e_{\alpha_{2n-1}}^{[1]}=0,\\
&&f_{\alpha_{n+1}}^{[1]}=\cdots=f_{\alpha_{2n-1}}^{[1]}=0,\eeqs
where $f_{\alpha_{n}}^{[1]}-e_{\alpha_{2n}}^{[1]}$ and $e_{\alpha_{n}}^{[1]}-f_{\alpha_{2n}}^{[1]}$ are root vectors of highest short root and lowest short root, respectively.

{\noindent\bf Case 3.} $\cl_1\cong so_{2n}$ and $\cl_k\cong sl_{2n}$ for all $2\leq k\leq K$.

Let $\{e_{\alpha_i}^{[1]}, f_{\alpha_i}^{[1]}|1\leq i\leq n\}$ be the analogue of  Chevalley generators of $C_n$ in $\cl_1$. and  $\{e_{\alpha_i}^{[k]}, f_{\alpha_i}^{[k]}|1\leq i\leq 2n-1\}$ be the analogue of Chevalley generators of $A_{2n-1}$ in $\cl_k$ for $k\geq 2$. Choose  a $K$-tuple ${\underline{a}}:=(-1,a_2,\cdots, a_K)\in(\mc^\times)^K$ such that $a_k\not=\pm1$ and $a_k\not=a_{j}^{\pm1}$ for all $k\not=j$. Set
\beqs e_{\alpha_{2n}}^{[k]}&=&a_k[f_{\alpha_{2n-1}}^{[k]},\cdots[f_{\alpha_2}^{[k]},f_{\alpha_1}^{[k]}]\cdots],\\
f_{\alpha_{2n}}^{[k]}&=&a_k^{-1}[\cdots[e_{\alpha_1}^{[k]},e_{\alpha_2}^{[k]}]\cdots, e_{\alpha_{2n-1}}^{[k]}],\eeqs
for all $2\leq k\leq K$, and
\beqs
e_{\alpha_{2n}}^{[1]}&=&f_{\alpha_n}^{[1]}-[\cdots[e_{\alpha_{n}}^{[1]},e_{\alpha_2}^{[1]}]\cdots, e_{\alpha_{n-1}}^{[1]}],\\
f_{\alpha_{2n}}^{[1]}&=&e_{\alpha_n}^{[1]}-[f_{\alpha_{n-1}}^{[1]},\cdots[f_{\alpha_2}^{[1]}, f_{\alpha_n}^{[1]}]\cdots],\\
&&e_{\alpha_{n+1}}^{[1]}=\cdots=e_{\alpha_{2n-1}}^{[1]}=0,\\
&&f_{\alpha_{n+1}}^{[1]}=\cdots=f_{\alpha_{2n-1}}^{[1]}=0.\eeqs

{\noindent\bf Case 4.} $\cl_1\cong so_{2n}$, $\cl_2\cong sp_{2n}$  and $\cl_k\cong sl_{2n}$ for all $3\leq k\leq K$.

Let $\{e_{\alpha_i}^{[1]}, f_{\alpha_i}^{[1]}|1\leq i\leq n\}$ be the analogue of  Chevalley generators of $D_n$ in $\cl_1$, $\{e_{\alpha_i}^{[2]}, f_{\alpha_i}^{[2]}|1\leq i\leq n\}$ be the analogue of  Chevalley generators of $C_n$ in $\cl_2$, and  $\{e_{\alpha_i}^{[k]}, f_{\alpha_i}^{[k]}|1\leq i\leq 2n-1\}$ be the analogue of Chevalley generators of $A_{2n-1}$ in $\cl_k$ for $k\geq 3$. Choose  a $K$-tuple ${\underline{a}}:=(-1,1, a_3,\cdots, a_K)\in(\mc^\times)^K$ such that $a_k\not=\pm1$ and $a_k\not=a_{j}^{\pm1}$ for all $k\not=j$. Set
\beqs e_{\alpha_{2n}}^{[k]}&=&a_k[f_{\alpha_{2n-1}}^{[k]},\cdots[f_{\alpha_2}^{[k]},f_{\alpha_1}^{[k]}]\cdots],\\
f_{\alpha_{2n}}^{[k]}&=&a_k^{-1}[\cdots[e_{\alpha_1}^{[k]},e_{\alpha_2}^{[k]}]\cdots, e_{\alpha_{2n-1}}^{[k]}],\eeqs
for all $3\leq k\leq K$, and
\beqs
e_{\alpha_{2n}}^{[1]}&=&f_{\alpha_n}^{[1]}-[\cdots[e_{\alpha_{n}}^{[1]},e_{\alpha_2}^{[1]}]\cdots, e_{\alpha_{n-1}}^{[1]}],\\
f_{\alpha_{2n}}^{[1]}&=&e_{\alpha_n}^{[1]}-[f_{\alpha_{n-1}}^{[1]},\cdots[f_{\alpha_2}^{[1]}, f_{\alpha_n}^{[1]}]\cdots],\\
&&e_{\alpha_{n+1}}^{[1]}=\cdots=e_{\alpha_{2n-1}}^{[1]}=0,\\
&&f_{\alpha_{n+1}}^{[1]}=\cdots=f_{\alpha_{2n-1}}^{[1]}=0,\\
e_{\alpha_{2n}}^{[2]}&=&f_{\alpha_n}^{[2]}-[\cdots[[e_{\alpha_{n}}^{[2]},e_{\alpha_1}^{[2]}],e_{\alpha_2}^{[2]}]\cdots, e_{\alpha_{n-1}}^{[2]}],\\
f_{\alpha_{2n}}^{[2]}&=&e_{\alpha_n}^{[2]}-[f_{\alpha_{n-1}}^{[2]},\cdots[f_{\alpha_2}^{[2]},[f_{\alpha_1}^{[2]},f_{\alpha_n}^{[2]}]]\cdots],\\
&&e_{\alpha_{n+1}}^{[2]}=\cdots=e_{\alpha_{2n-1}}^{[2]}=0,\\
&&f_{\alpha_{n+1}}^{[2]}=\cdots=f_{\alpha_{2n-1}}^{[2]}=0.\eeqs

In all above four cases, we define a homomorphism from $\gim(M_n)\rightarrow\cl$ via
\beqs \Psi_{\underline{a}}:&& e_i\mapsto \sum_{k=1}^K\big(e_{\alpha_i}^{[k]}-f_{\alpha_{n+i}}^{[k]}\big), \; f_i\mapsto \sum_{k=1}^K\big(f_{\alpha_i}^{[k]}-e_{\alpha_{n+i}}^{[k]}\big),\eeqs
for all $1\leq i\leq n$.

\begin{proposition}
$\Psi_{\underline{a}}$ is a Lie algebra epimorphism.
\end{proposition}

\begin{proof}
  Let $P_k$ be the projection from $\cl$ to $\cl_k$.  By Lemmas \ref{L2.1}-\ref{L2.3}, we infer that $P_k\circ \Psi_{\underline{a}}$ is a homomorphism for all $1\leq k\leq K$, then $\Psi_{\underline{a}}=\oplus_{k=1}^K P_k\circ\Psi_{\underline{a}}$ is a homomorphism of Lie algebras.
 The detailed proof for that $\Psi_{\underline{a}}$ is an epimorphism, is very similar to the proof for Lemmas 5.1-5.4 in Section 5.
\end{proof}

\section{Main results}

Our main result can be stated by the following three theorems.

\begin{theorem}
If $\cl$ is of type $M(n, {\bf a}, {\bf c}, {\bf d})$, then there exists an  ideal $I$ of $\gim(M_n)$, such that
\beqs \gim(M_n)/I\cong \cl.\eeqs
\end{theorem}

\begin{theorem}
If there exists an ideal $I$ of $\gim(M_n)$, such that $\gim(M_n)/I$ is finite-dimensional and semi-simple, then $\gim(M_n)/I$ is   of type $M(n, {\bf a}, {\bf c}, {\bf d})$.
\end{theorem}

The above two theorems will be proved in below sections.

\begin{theorem}
(1) If $V$ is an irreducible finite-dimensional module of an $M(n, {\bf a}, {\bf c}, {\bf d})$ type  Lie algebra, then $V$ is an irreducible $\gim(M_n)$-module.

(2) If $V$ is an irreducible $\gim(M_n)$-module with finite dimension, then $V$ is an irreducible module of  some $M(n, {\bf a}, {\bf c}, {\bf d})$ type Lie algebra.
\end{theorem}
\begin{proof}

(1)  Suppose that $V$ is an irreducible finite-dimensional module of $M(n, {\bf a}, {\bf c}, {\bf d})$ type  Lie algebra $\cl$,  and that $\zeta:\cl\rightarrow End(V)$ is the representation map. By Proposition 2.5,  $V$ is an irreducible $\gim(M_n)$-module with representation map $\zeta\circ\Psi_{\underline{a}}$.

(2)  Suppose that $V$ is non-trivial and  the representation is define by $\kappa: \gim(M_n)\rightarrow End(V)$ and the image of $\kappa$ has Levi decomposition
$$Im(\kappa)=S\dot+H\dot+W,$$
where $S$ is semi-simple, $H$ is diagonal and $W$ is nilpotent. Undoubtedly, it holds that $[S, H]=0$.

Because  $V$ is non-trivial and there exists a positive integer $k$ such that $W^k$ trivially acts on $V$,  we infer that $W\cdot V$ is a proper submodule and we may assume that $W=0$.

Suppose that $\sh$ is a Cartan subalgebra of $S$. Then $V$ has a basis such that both $\kappa(\sh)$ and $\kappa(H)$ are diagonal, then any $S$-module must be an $H$-module. So $V$ is irreducible as $\gim(M_n)$-module if and only if $V$ is irreducible as $S$-module. By Theorem 3.2, $S$ has to be isomorphic to some  $M(n, {\bf a}, {\bf c}, {\bf d})$ type  Lie algebra, and thus the statement holds.
\end{proof}

\section{Evaluation representations of $A_{2n-1}^{(1)}$}

Let $\sg$ be the  simple Lie algebra $sl_{2n}(\mc)$. Then the affine Lie algebra of type $A_{2n-1}^{(1)}$ has a realization:
$$\hat{\sg}=\sg\otimes\mc[t,t^{-1}]\oplus\mc c,$$
with  bracket:
$$[c,\hat{\sg}]=0,\quad [x\otimes t^m, y\otimes t^k]=[x,y]\otimes t^{m+k}+m\delta_{m,-k}(x,y)c,$$
where $x,y\in\sg, m,k\in\mz$ and $(,)$ is a non-degenerate invariant form on $\sg$, which is a scalar of the Killing form.

The following is the Dynkin diagram of $\sg$:

\centerline{\xymatrix{^1\circ\ar@{-}[r]&^2\circ\ar@{-}[r]&^3\circ-&-\circ\ar@{-}[r]&\circ^{2n-1}}}
Let $\Pi=\{\alpha_1,\cdots, \alpha_{2n-1}\}$ be the associated prime root system,  and $\dot\Delta$ be its root system.  Set $\{E_\alpha, H_i|\alpha\in\dot\Delta, 1\leq i\leq 2n-1\}$ be the Chevalley basis of $\sg$. Then elements $E_{\pm\alpha_i}\otimes 1(1\leq i\leq 2n-1)$ and $E_{\pm\alpha^\flat}\otimes t^\mp$ are Chevalley generators of $\hat{\sg}$, where $\alpha^\flat=\alpha_1+\cdots+\alpha_{2n-1}$.

Set
\beqs &\hat e_1=E_{\alpha_1}\otimes 1-E_{-\alpha_{n+1}}\otimes 1,\quad \hat f_1=E_{-\alpha_1}\otimes 1-E_{\alpha_{n+1}}\otimes 1, \quad\hat h_1=H_1\otimes 1-H_{n+1}\otimes 1, \\
&\hat e_2=E_{\alpha_2}\otimes 1-E_{-\alpha_{n+2}}\otimes 1,\quad  \hat f_2=E_{-\alpha_2}\otimes 1-E_{\alpha_{n+2}}\otimes 1, \quad \hat h_2=H_2\otimes 1-H_{n+2}\otimes 1, \\
&\cdots\\
&\hat e_n=E_{\alpha_n}\otimes 1+E_{\alpha^\flat}\otimes t^{-1}, \quad \hat f_n=E_{-\alpha_n}\otimes 1+E_{-\alpha^\flat}\otimes t,\quad \hat h_n=H_n\otimes 1+(H_1+\cdots+H_{2n-1})\otimes 1-c, \eeqs
and
let $\hat{\sg}_{fp}$ be the subalgebra generated by $\{\hat e_i,\hat f_i|1\leq i\leq n\}$. $\hat{\sg}_{fp}$ is called a fixed point subalgebra of $\hat{\sg}$.

\begin{proposition}\label{L4.1}
There exists an algebra homomorphism $\phi: \cu(\gim(M_n))\rightarrow \cu(\hat{\sg}_{fp})$ via the action on generators:
  \beqs \phi(e_i)=\hat e_i,\quad  \phi(f_i)=\hat f_i,\quad 1\leq i\leq n.\eeqs
  More over, it induces  a Lie algebra homomorphism (also denoted by $\phi$) from $\gim(M_n)$ to $\hat{\sg}_{fp}$.
\end{proposition}
 \begin{proof}
 This is a consequence of the following equations:
 \beqs
 [\hat e_{i}, \hat f_{i}]=\hat h_i, &&i=1,\cdots,n,\\
 {[\hat h_i, \hat e_{j}]}=m_{i,j}\hat e_{j},  [\hat h_i, \hat f_{j}]=-m_{i,j}\hat f_{j}, &&1\leq i,j\leq n,\\
 {[\hat e_i,\hat f_j,]}=0,&&i\not=j, \{i,j\}\not=\{1,n\},\\
 {[\hat e_i, [\hat e_i, \hat e_{j}]]}=0,  {[\hat f_i, [\hat f_i, \hat f_{j}]]}=0,&& |i-j|=1,\\
 {[\hat e_i, \hat e_{j}]}=0,  {[\hat f_i, \hat f_{j}]}=0,&& |i-j|>1,\\
  {[\hat e_i,[\hat e_i,\hat f_j,]]}=[\hat f_i,[\hat f_i,\hat e_j,]]=0,&& \{i,j\}=\{1,n\}. \eeqs
 \end{proof}

{\noindent\bf Remark 2.} {\it In fact, the map $\phi$ is an isomorphism (see \cite{Br}). More over, if we set
\beqs \hat e_n=E_{\alpha_n}\otimes 1-E_{\alpha^\flat}\otimes t^{-1}, \quad \hat f_n=E_{-\alpha_n}\otimes 1-E_{-\alpha^\flat}\otimes t, \eeqs
then Proposition 4.1 still holds.  However, these two maps are equivalent under the automorphism of $\hat{sl}_{2n}$ induced by $x\otimes t^m\mapsto x\otimes (-t)^m$. }

Let $a\in\mc^\times$, an  irreducible evaluation $\hat{\sg}$-module $V_{\lambda,a}$ is an irreducible highest weight $\sg$-module with highest weight $\lambda$, and the action is defined by:
\beqs (x\otimes t^m)\cdot v=a^mx\cdot v, &c\cdot v=0.\eeqs
for all $x\in\sg, m\in\mz$ and $v\in V_{\lambda,a}$.

For any $V_{\lambda,a}$,  the  map $\phi$ induces a representation for $\gim({M_n})$. Particularly, for any dominate weight $\lambda$, $V_{\lambda,a}$ is finite-dimensional.

Let $V_{\lambda,a}$ be the $2n$-dimensional  irreducible module, i.e., the natural representation of $\sg$. Particularly, we may assume that $\sg=sl_{2n}$. The
following proposition is an important tool of the proof for our main results.

\begin{proposition}\label{P4.2}
The following map induces a representation for $\gim(M_n)$:
\beqs &\psi_a: &\gim(M_n)\rightarrow sl_{2n}\\
&&\psi_a(e_1)=E_{1,2}-E_{n+2,n+1},\quad\psi_a(f_1)=E_{2,1}-E_{n+1,n+2}\\
&&\psi_a(e_2)=E_{2,3}-E_{n+3,n+2},\quad\psi_a(f_2)=E_{3,2}-E_{n+2,n+3}\\
&&\cdots\\
&&\psi_a(e_{n-1})=E_{n-1,n}-E_{2n,2n-1},\quad\psi_a(f_{n-1})=E_{n,n-1}-E_{2n-1,2n}\\
&&\psi_a(e_n)=E_{n,n+1}+a^{-1}E_{1,2n},\quad\psi_a(f_{n})=E_{n+1,n}+aE_{2n,1}.\eeqs
Particularly, if $a\not=\pm1$, the image of $\psi_a$ is $sl_{2n}$, if $a=1$, the image of $\psi_a$ is $sp_{2n}$, and if $a=-1$, then the image of $\psi_a$ is $so_{2n}$.
\end{proposition}

{\noindent\bf Remark 3.}  {\it In case of assuming that
\beqs \hat e_n=E_{\alpha_n}\otimes 1-E_{\alpha^\flat}\otimes t^{-1}, \quad \hat f_n=E_{-\alpha_n}\otimes 1-E_{-\alpha^\flat}\otimes t, \eeqs
Proposition 4.2 still holds after exchanging the positions of results when $a=1$ and $a=-1$, respectively.
}

\begin{proof} \quad {\bf Proof for Proposition \ref{P4.2}.}

Set $x_i=\psi_a(e_i), y_i=\psi_a(f_i)$ for $1\leq i\leq n$, and $h_i=E_{i,i}-E_{i+1,i+1}-E_{i+n,i+n}+E_{i+1+n,i+1+n}$ for $1\leq i\leq n-1$. Then it is easy to know that the elements $x_i,y_i (1\leq i\leq n-1)$ generate the Lie algebra
$$S=\left\{X=\left[\begin{array}{cc}A&0\\0&-A^t\end{array}\right]\Big| A\in sl_n\right\},$$
which is isomorphic to $sl_n$.

More over,
\beqs [x_{n-1},x_n]&=&E_{n-1, n+1}+a^{-1}E_{1, 2n-1},\\
{[x_{n-2},[x_{n-1},x_n]]}&=&E_{n-2, n+1}+a^{-1}E_{1, 2n-2},\\
{[x_{n-3},[x_{n-2},[x_{n-1},x_n]]]}&=&E_{n-3, n+1}+a^{-1}E_{1, 2n-3},\\
\cdots&\cdots&\cdots\\
{[x_2\cdots,[x_{n-1},x_n]\cdots]}&=&E_{2, n+1}+a^{-1}E_{1, n+2},\\
{[x_1,[x_2\cdots,[x_{n-1},x_n]\cdots]]}&=&(1+a^{-1})E_{1,n+1},\eeqs
\beqs [y_{n},y_{n-1}]&=&E_{n+1, n-1}+aE_{2n-1,1},\\
{[[y_{n},y_{n-1}],y_{n-2}]}&=&E_{n+1, n-2}+aE_{2n-2,1},\\
{[[[y_{n},y_{n-1}],y_{n-2}], y_{n-3}]}&=&E_{n+1, n-3}+aE_{2n-3,1},\\
\cdots&\cdots&\cdots\\
{[\cdots[y_{n},y_{n-1}],\cdots y_2]}&=&E_{n+1, 2}+aE_{n+2,1},\\
{[[\cdots[y_{n},y_{n-1}],\cdots y_2],y_1]}&=&(1+a)E_{n+1,1},\eeqs
then we divide the proof into three cases according to the value of $a$.

(1) {\it Case 1.}\quad  $a\not=\pm1$.

By equations
\beqs [x_1,[x_2\cdots,[x_{n-1},x_n]\cdots]]&=&(1+a^{-1})E_{1,n+1},\\
{[[\cdots[y_{n},y_{n-1}],\cdots y_2],y_1]}&=&(1+a)E_{n+1,1},\eeqs
we have $E_{1,n+1}, E_{n+1,1}\in Im(\psi_a)$, and hence, for every $1\leq i\leq n-1$,
\beqs 2E_{i+1,n+i+1}=[y_i,[y_i, E_{i,n+i}]],\\
2E_{n+i+1,i+1}=[x_i,[x_i, E_{n+i,i}]],\eeqs
thus it holds that
$$E_{i,n+i}, E_{n+i,i}\in Im(\psi_a), \quad \forall 1\leq i\leq n.$$

Set
 \beqs {\bf x}_k=E_{n+k,k}&&{\bf y}_k=E_{k,n+k},\eeqs
then
$$[x_{n}, {\bf x}_n]=E_{2n,n+1}-a^{-1}E_{1,n},\; [y_{n}, {\bf y}_n]=E_{n+1,2n}-aE_{n,1}$$
this implies that $E_{2n, n+1}, E_{n+1,2n}\in Im(\psi_a)$, and hence, $E_{1,n}, E_{n,1}\in Im(\psi_a)$.

Further, the iteration formula
\beqs [y_i, E_{i,n}]=E_{i+1,n}, &&[E_{n,i}, x_i]=E_{n,i+1},\eeqs
 implies that $E_{n,i}, E_{i,n}\in Im(\psi_a)$ for all $i\leq n$, and thus $E_{n+i,2n}, E_{2n,n+i}\in Im(\psi_a)$ for all $i\leq n$.

We also have iteration formula
\beqs [E_{i,n}, y_{n-1}]=E_{i,n-1}, [x_{n-1}, E_{n,i}]=E_{n-1,i},\\
{[E_{i,k+1}, y_{k}]}=E_{i,k}, [x_{k}, E_{k+1,i}]=E_{k,i}, \hbox{\rm\; if\;} i<k,
\eeqs
and this implies that $E_{i,j}, E_{j,i}\in Im(\psi_a)$ for all $i<j\leq n$. Finally, we infer that $E_{i,j}, E_{i+n,j+n}\in Im(\psi_a)$ for all $1\leq i,j\leq n$ and $i\not=j$. Notice that $$[[E_{2,n}, E_{n,2}], x_n]=-E_{n,n+1},\; [[E_{2,n}, E_{n,2}], y_n]=E_{n+1,n},$$
so we have $E_{i,i+1}, E_{i+1,i}\in Im(\psi_a)$ for all $1\leq i\leq 2n-1$, this is a set of Chevalley generators of $sl_{2n}$, and then $Im(\psi_a)=sl_{2n}$.

(2) {\it Case 2.}\quad  $a=1$.

It is easy to check that
\beqs [{\bf x}_1, y_i]=[{\bf y}_1, x_i]=0&&i=1,\cdots,n-1,\\
{[{\bf x}_1, x_i]}=[{\bf y}_1, y_i]=0&&i=2,\cdots,n-1,\\
{[[{\bf x}_1, x_1],x_1]}=2E_{n+2,2}=2{\bf x}_2,&& [[{\bf y}_1, y_1],y_1]=2E_{2,n+2}=2{\bf y}_2,\\
{[{\bf x}_2, x_1]}=[{\bf y}_2, y_1]=0, &&{[{\bf x}_1,[{\bf x}_1, x_1]]}=[{\bf y}_1, [{\bf y}_1, y_1]]=0,
\eeqs
then the Lie algebra generated by $S$ and ${\bf x}_1, {\bf y}_1$ is isomorphic to $sp_{2n}$. In particular, $\{x_i, y_i, {\bf x}_1, {\bf y}_1|1\leq i\leq n-1\}$ is a set of Chevalley generators. The associated Dynkin diagram is

\centerline{\xymatrix{^n\circ\ar@{=>}[r]&^1\circ\ar@{-}[r]&^2\circ\ar@{-}[r]&^3\circ-&-\circ\ar@{-}[r]&\circ^{n-1}}}
Notice that $x_i, y_i\in sp_{2n}$ for all $1\leq i\leq n$, we have
\beqs \langle S, {\bf x}_1, {\bf y}_1\rangle \subset Im(\psi_a)=\langle x_i, y_i |1\leq i\leq n\rangle \subset sp_{2n},\eeqs
the dimension relation $\dim\langle S, {\bf x}_1, {\bf y}_1\rangle=\dim sp_{2n}$ implies that $Im(\psi_a)=sp_{2n}$.

(3) {\it Case 3.}\quad  $a=-1$.
Then $x_n, y_n$ have the form $\left[\begin{array}{cc}0&B\\0&0\end{array}\right]$ and $\left[\begin{array}{cc}0&0\\B&0\end{array}\right]$, respectively, and where $B=-B^t\in gl_n$ is anti-symmetric.

Set
\beqs x^\sharp&=&[\cdots[y_{n},y_{n-1}],\cdots y_2]=E_{n+1,2}-E_{n+2,1},\\
y^\sharp&=&[x_2,\cdots[x_{n-1},x_{n}]\cdots]=E_{2, n+1}-E_{1, n+2},\eeqs
we have
\beqs   [{\bf x}^\sharp, y_i]=[{\bf y}^\sharp, x_i]=0,&&i=1,\cdots,n-1,\\
{[{\bf x}^\sharp, x_i]}=[{\bf y}^\sharp, y_i]=0,&&i=1,3,\cdots,n-1,\eeqs
and the set $\{{\bf x}^\sharp, {\bf y}^\sharp, x_2, y_2\}$ generates a simple Lie algebra of type $A_2$.

Then the Lie algebra generated by $S$ and ${\bf x}^\sharp, {\bf y}^\sharp$ is isomorphic to $so_{2n}$. In particular, $\{x_i, y_i, {\bf x}^\sharp, {\bf y}^\sharp|1\leq i\leq n-1\}$ is a set of Chevalley generators.
Notice that $x_i, y_i\in so_{2n}$ for all $1\leq i\leq n$, we have
\beqs \langle S, {\bf x}^\sharp, {\bf y}^\sharp\rangle \subset Im(\psi_a)=\langle x_i, y_i |1\leq i\leq n\rangle \subset so_{2n},\eeqs
the dimension relation $\dim\langle S, {\bf x}^\sharp, {\bf y}^\sharp\rangle=\dim so_{2n}$ implies that $Im(\psi_a)=so_{2n}$.
\end{proof}

\section{Proof for Theorem 3.1}


Let $\ca=\oplus_{k=1}^K\sg_i=\sg^{\oplus n}$ be the direct sum of $K$ copies of $\sg=sl_{2n}$.  Fix a $K$-tuple $\underline{a}=(a_1,\cdots,a_K)\in(\mc^\times)^K$ such that $a_k\not=a_j^{\pm1}$ for all $k\not=j$. Define the map
\beqs \psi_{\underline{a}}=\bigoplus_{k=1}^K\psi_{a_k}:&&\gim(M_n)\rightarrow\ca,\eeqs
where $\psi_{a_k}:\gim(M_n)\rightarrow\sg_k$ is an analogue of the evaluation map defined in Proposition \ref{P4.2}..

{\it Theorem 3.1 can be proved by the following Lemmas \ref{L5.1}-\ref{L5.4}. }

\begin{lemma}\label{L5.1}
If $a_k\not=\pm1$ for every $k$, then $\psi_{\underline{a}}$ is an epimorphism.
\end{lemma}
\begin{proof}
For convenience, we let $\psi(x)^{[k]}$ denote the image of $x$ under $\psi_{a_k}$.
By the definition of $\psi$, we have that
\beqs
&&\psi_{a_k}(e_1)=(E_{1,2}-E_{n+2,n+1})^{[k]},\quad\psi_{a_k}(f_1)=(E_{2,1}-E_{n+1,n+2})^{[k]}\\
&&\psi_{a_k}(e_2)=(E_{2,3}-E_{n+3,n+2})^{[k]},\quad\psi_{a_k}(f_2)=(E_{3,2}-E_{n+2,n+3})^{[k]}\\
&&\cdots\\
&&\psi_{a_k}(e_{n-1})=(E_{n-1,n}-E_{2n,2n-1})^{[k]},\quad\psi_{a_k}(f_{n-1})=(E_{n,n-1}-E_{2n-1,2n})^{[k]}\\
&&\psi_{a_k}(e_n)=(E_{n,n+1}+a_k^{-1}E_{1,2n})^{[k]},\quad\psi_{a_k}(f_{n})=(E_{n+1,n}+a_kE_{2n,1})^{[k]}.\eeqs
Set
\beqs x_i^{[k]}=\psi_{a_k}(e_i),&& y_i^{[k]}=\psi_{a_k}(f_i).\eeqs Then
\beqs [x_{n-1}^{[k]},x_n^{[k]}]&=&(E_{n-1, n+1}+a_k^{-1}E_{1, 2n-1})^{[k]},\\
{[x^{[k]}_{n-2},[x^{[k]}_{n-1},x^{[k]}_n]]}&=&(E_{n-2, n+1}+a_k^{-1}E_{1, 2n-2})^{[k]},\\
\cdots&\cdots&\cdots\\
{[x_1^{[k]},[x_2^{[k]}\cdots,[x_{n-1}^{[k]},x_n^{[k]}]\cdots]]}&=&(1+a_k^{-1})(E_{1,n+1})^{[k]},\eeqs
\beqs [y_{n}^{[k]},y^{[k]}_{n-1}]&=&(E_{n+1, n-1}+a_k E_{2n-1,1})^{[k]},\\
{[[y_{n}^{[k]},y_{n-1}^{[k]}],y_{n-2}^{[k]}]]}&=&(E_{n+1, n-2}+a_kE_{2n-2,1})^{[k]},\\
\cdots&\cdots&\cdots\\
{[[\cdots[y_{n}^{[k]},y_{n-1}^{[k]}],\cdots y_2^{[k]}],y_1^{[k]}]}&=&(1+a_k)(E_{n+1,1})^{[k]},\eeqs
we infer that
\beqs H_{1,n+1}&:=&\Big[\sum_{k=1}^K(1+a_k^{-1})(E_{1,1+n})^{[k]},\sum_{k=1}^K(1+a_k)(E_{n+1,1})^{[k]}\Big]\\
&=&\sum_{k=1}^K[(1+a_k^{-1})(E_{1,1+n})^{[k]},(1+a_k)(E_{n+1,1})^{[k]}]\\
&=&\sum_{k=1}^K(2+a_k+a_k^{-1})(E_{1,1}-E_{n+1,n+1})^{[k]}.
\eeqs
Consider its action on $\psi_{\underline{a}}(e_1)=\sum_{k=1}^Kx_1^{[k]}$ and $\psi_{\underline{a}}(f_1)=\sum_{k=1}^Ky_1^{[k]}$,  we have that
\beqs (\ad H_{1,n})(\psi_{\underline{a}}(e_1))&=&\sum_{k=1}^K(2+a_k+a_k^{-1})x_1^{[k]},\\
(-\ad H_{1,n})(\psi_{\underline{a}}(f_1))&=&\sum_{k=1}^K(2+a_k+a_k^{-1})y_1^{[k]},\\
(\ad H_{1,n})^2(\psi_{\underline{a}}(e_1))&=&\sum_{k=1}^K(2+a_k+a_k^{-1})^2x_1^{[k]},\\
(-\ad H_{1,n})^2(\psi_{\underline{a}}(f_1))&=&\sum_{k=1}^K(2+a_k+a_k^{-1})^2y_1^{[k]},\\
\cdots&\cdots&\cdots\\
(\ad H_{1,n})^K(\psi_{\underline{a}}(e_1))&=&\sum_{k=1}^K(2+a_k+a_k^{-1})^Kx_1^{[k]},\\
(-\ad H_{1,n})^K(\psi_{\underline{a}}(f_1))&=&\sum_{k=1}^K(2+a_k+a_k^{-1})^Ky_1^{[k]}.\eeqs

If $2+a_k+a_k^{-1}=2+a_j+a_j^{-1}$, then $(a_ka_j-1)(a_k-a_j)=0$. By our assumption, every $2+a_k+a_k^{-1}$ is non-zero and $2+a_k+a_k^{-1}\not=2+a_j+a_j^{-1}$ for all $k\not=j$. By the invertibility of Vandermonde matrix, we infer that
\beqs x_1^{[1]}, y_1^{[1]}, x_1^{[2]},y_1^{[2]}\cdots,x_1^{[K]},y_1^{[K]}\in Im(\psi_{\underline{a}}).\eeqs

Let $H_{1,2}^{[k]}=[x_1^{[k]},y_1^{[k]}]=(E_{1,1}-E_{2,2})^{[k]}$, and consider its action on $\psi_{\underline{a}}(e_2), \psi_{\underline{a}}(f_2)$, we infer that
\beqs x_2^{[1]}, y_2^{[1]}, x_2^{[2]},y_2^{[2]}\cdots,x_2^{[K]},y_2^{[K]}\in Im(\psi_{\underline{a}}).\eeqs
Repeat this process, similar argument implies that
\beqs x_i^{[k]},y_i^{[k]}\in Im(\psi_{\underline{a}})\eeqs
for every $k$.
By Proposition \ref{P4.2},  we can infer that $\sg_k\in Im(\psi_{\underline{a}})$ for all $k$, and thus the map $\psi_{\underline{a}}$ is surjective.
\end{proof}

\begin{lemma}\label{L5.2}
If $a_1=1$ and $a_k\not=-1$ for every $k\geq 2$, then $Im(\psi_{\underline{a}})$ is a direct sum of one copy of $sp_{2n}$ and $K-1$ copies of $sl_{2n}$.
\end{lemma}
\begin{proof}
We may repeat the proof for Lemma \ref{L5.1} until we obtain the result
\beqs x_i^{[k]},y_i^{[k]}\in Im(\psi_{\underline{a}}),\; \forall 1\leq k\leq K, 1\leq i\leq n,\eeqs
then $Im(\psi_{\underline{a}})=\bigoplus_{k=1}^K Im(\psi_{a_k})$, by Proposition \ref{P4.2}, $Im(\psi_{a_1})$ is isomorphic to $sp_{2n}$, and $Im(\psi_{a_k})$ is isomorphic to $sl_{2n}$ for $k\geq 2$.
\end{proof}

\begin{lemma}\label{L5.3}
If $a_1=-1, a_2=1$, then $Im(\psi_{\underline{a}})$ is a direct sum of one copy of $so_{2n}$, one copy of $sp_{2n}$ and $K-2$ copies of $sl_{2n}$.
\end{lemma}
\begin{proof}
We may repeat the proof for Lemma \ref{L5.1} until we obtain $H_{1,n+1}$. Notice that, in this case
\beqs H_{1,n+1}&:=&\sum_{k=2}^K(2+a_k+a_k^{-1})(E_{1,1}-E_{n+1,n+1})^{[k]}.\eeqs
Then go on with the process in Lemma \ref{L5.2}, we get that
\beqs x_i^{[k]},y_i^{[k]}\in Im(\psi_{\underline{a}}), \; \forall 2\leq k\leq K, 1\leq i\leq n,\eeqs
then
\beqs x_i^{[1]}=\psi_{\underline{a}}(e_i)-\big(\sum_{k=2}^Kx_i^{[k]}\big)\in Im(\psi_{\underline{a}}),&&y_i^{[1]}=\psi_{\underline{a}}(f_i)-\big(\sum_{k=2}^Ky_i^{[k]}\big)\in Im(\psi_{\underline{a}}),\eeqs
it also holds that $Im(\psi_{\underline{a}})=\bigoplus_{k=1}^K Im(\psi_{a_k})$. By Proposition \ref{P4.2}, $Im(\psi_{a_1})$ is isomorphic to $so_{2n}$, $Im(\psi_{a_2})$ is isomorphic to $sp_{2n}$, and $Im(\psi_{a_k})$ is isomorphic to $sl_{2n}$ for $k\geq 3$.
\end{proof}

\begin{lemma}\label{L5.4}
If $a_1=-1$ and $a_k\not=1$ for $k\geq 2$, then $Im(\psi_{\underline{a}})$ is a direct sum of one copy of $so_{2n}$ and $K-1$ copies of $sl_{2n}$.
\end{lemma}
\begin{proof}
The proof is very similar to that for Lemma \ref{L5.3}.
\end{proof}

\section{Proof for   Theorem 3.2}

If $G(t)$ is a polynomial with non-zero constant term, we may define  the quotient algebra
\beqs \hat{\sg}/\langle\sg\otimes G(t)\mc[t,t^{-1}]\rangle,\eeqs
and denoted by $\hat{\sg}/G(t)$.  Note that, we may regard the central element $c$ as $0$ in the quotient algebra $\hat{\sg}/G(t)$ for all non-trivial polynomial $G(t)$.

\begin{lemma}
If $G(t)|F(t)$, then $\hat{\sg}/G(t)$ is a quotient of $\hat{\sg}/F(t)$. Particularly,
\beqs \hat{\sg}/G(t)\cong \frac{\hat{\sg}/F(t)}{\sg\otimes G(t)\mc[t,t^{-1}]/F(t)}.\eeqs
\end{lemma}
\begin{proof}
This result follows from the homomorphism fundamental theorem.
\end{proof}

\begin{lemma}
If $(x-a)^r|G(t)$ for some $r>1$, then $\hat{\sg}/G(t)$ has a nilpotent ideal
$$I=\sg\otimes\frac{G(t)}{(t-a)}\mc[t,t^{-1}].$$
\end{lemma}
\begin{proof}
Obviously, for any $x\in\hat{\sg}$ and $y\in I$, we have $[x,y]\in I$. Suppose that $r^*\geq r$ is the multiple of prime factor $t-a$, then for any $u\otimes f(t), v\otimes g(t)\in I$, we have $(t-a)^2|G(t)$, and
$$G(t)\big|\frac{G(t)^2}{(t-a)^2},\quad \frac{G(t)^2}{(t-a)^2}\big|f(t)g(t),$$and thus $[I,I]=0$.
\end{proof}

If $G(t)$ has degree $K$, then it is well-known  that $\hat{\sg}/G(t)$ is the direct sum of $K$ copies of $\sg$ if $G(t)$ has $K$ many different factors $t-a_i$, where $a_i\not=0 (1\leq i\leq K)$.

Let $G(t)=\prod_{i=1}^k(t-a_i)$, then there exists $c_i\not=0$, such that
\beqs \sum_{i=1}^k c_i\frac{G(t)}{t-a_i}=1,\eeqs
then
\beqs \hat{\sg}/G(t)=\bigoplus_{i=1}^k\hat{\sg}\otimes \frac{G(t)}{t-a_i}/G(t)\cong \bigoplus_{i=1}^k\frac{\hat{\sg}}{t-a_i}.\eeqs

Let $\phi$ be as defined in Section 4. Then
\beqs [\hat{e}_1,\cdots[\hat{e}_{n-1},\hat{e}_{n}]\cdots ]&=&E_{\alpha_1+\cdots+\alpha_n}\otimes(1+t^{-1}),\\
{[\cdots[\hat{f}_{n-1},\hat{f}_{n-2}],\cdots\hat{f}_1]}&=&E_{-\alpha_1-\cdots-\alpha_{n-1}}\otimes1-E_{\alpha_{n+1}+\cdots+\alpha_{2n-1}}\otimes1,\\
{[[\cdots[\hat{f}_{n-1},\hat{f}_{n-2}],\cdots\hat{f}_1],E_{\alpha_1+\cdots+\alpha_n}\otimes(1+t^{-1})}]&=&(E_{\alpha_n}-E_{\alpha_{1}+\cdots+\alpha_{2n-1}})\otimes(1+t^{-1}),\\
{[(E_{\alpha_n}-E_{\alpha_{1}+\cdots+\alpha_{2n-1}})\otimes(1+t^{-1}), \hat{f}_n]}&=&H_n\otimes(1+t^{-1})+(H_1+\cdots H_{2n-1})\otimes (1+t)-c,
\eeqs
this implies that
\beqs \Xi:=H_n\otimes t^{-1}+(H_1+\cdots+H_{2n-1})\otimes t\in \phi(\gim(M_n)),\eeqs
then
\beqs [(\frac12\ad \Xi)^m(\hat{e}_n),\hat{f}_n]&=&H_n\otimes t^{-m}+(H_1+\cdots+H_{2n-1})\otimes t^m,\\
{[(\ad \Xi)^m(\hat{e}_1),\hat{f}_1]}&=&H_1\otimes t^{m}-H_{n+1}\otimes t^{-m},\eeqs
for all $m>0$. Similarly, we infer that
\beqs H_i\otimes t^{-m}-H_{n+i}\otimes t^m\in \phi(\gim(M_n)),\eeqs
for all $1\leq i\leq n$ and $m\in\mz\setminus\{0\}$, where $H_{2n}=-H_1-\cdots-H_{2n-1}$.

Moreover,
\beqs (H_i\otimes t^{-m}-H_{n+i}\otimes t^m)\pm(H_i\otimes t^{m}-H_{n+i}\otimes t^{-m})&=&(H_i\mp H_{n+i})\otimes (t^m\pm t^{-m}).\eeqs

Suppose that $I$ is an ideal of $\gim(M_n)$ such that the quotient algebra $\gim(M_n)/I$ is finite-dimensional and semi-simple. Since $\phi$ is an isomorphism, we have that $\phi(I)$ is an ideal of $\hat{\sg}_{fp}=\phi(\gim(M_n))\cong\gim(M_n)$ and $\hat{\sg}_{fp}/\phi(I)\cong\gim(M_n)/I$. Note that
every polynomial in variable $t+t^{-1}$ is a linear combination of $t^{m}+t^{-m} (m\geq 0)$. Then there must hold that
\beqs (H_i-H_{n+i})\otimes \theta(t)\in\phi(I),&&1\leq i\leq n,\eeqs
where $\theta(t)=\eta(t+t^{-1})$ for some polynomial $\eta$. Otherwise, the elements
\beqs \Big\{H\otimes(t^j+t^{-j})| j\in\mz_{\geq0}\Big\}\eeqs
are linearly independent in $\hat{\sg}_{fp}/\phi(I)$ for some $H\in span\{H_i-H_{n+i}|1\leq i\leq n\}$,  which contradicts to the assumption of the finite-dimension.

Actually, $\hat{h}_i=H_i-H_{n+i}$ for all $1\leq i\leq n$. Let $\sh_{fp}=span\{\hat{h}_i|1\leq i\leq n\}$. Then  $\hat{\sg}_{fp}$ has a root decomposition
\beqs \hat{\sg}_{fp}=(\hat{\sg}_{fp})_0+\sum_{\alpha}(\hat{\sg}_{fp})_\alpha,\eeqs
where $(\hat{\sg}_{fp})_0=\{x\in \hat{\sg}_{fp}| [\sh_{fp},x]=0\}$ and $(\hat{\sg}_{fp})_\alpha=\{x\in \hat{\sg}_{fp}|[h,x]=\alpha(h)x,\forall h\in\sh_{fp}\}$.

By Berman's result (\cite{Br}, Proposition 1.12), it holds that $\hat{\sg}=\hat{\sg}_{fp}\oplus M$, where $M$ is a $\hat{\sg}_{fp}$-module and $[M,M]\subseteq \hat{\sg}_{fp}$.
If $x=\sum_{j=1}^d x_j\otimes p_j(t)\theta(t)\in (\hat{\sg}_{fp})_\alpha, \alpha\not=0$,
then there exists $H\in\sh_{fp}$ such that $\alpha(H)\not=0$ and
\beqs x=\frac{1}{\alpha(H)}\sum_{j=1}^d [H\otimes\theta(t), x_j\otimes p_j(t)]\in \phi(I).\eeqs
Otherwise, we infer that $\sum_{j=1}^d x_j\otimes p_j(t)\not\in\hat{\sg}_{fp}$ and thus $x\not\in \hat{\sg}_{fp}$, which is a contradiction.

Moreover, $(\hat{\sg}_{fp})_0\subseteq \sh\otimes\mc[t,t^{-1}]\oplus\mc c$ and
\beqs (\sh\otimes\mc[t+t^{-1}]\oplus\mc c)\bigcap\hat{\sg}_{fp}&=&\sum_{m=0}^\infty\sum_{i=1}^n\mc\big((H_i-H_{n+i})\otimes (t^m+t^{-m})-\delta_{i,n}c\big),\eeqs
then we have the following result:
\begin{lemma}
We let $J$ be the ideal of $\hat{\sg}$ generated by  $\{H_i\otimes\theta(t) |1\leq i\leq 2n-1\}$, then
\beqs\phi(\gim(M_n))\bigcap J\subseteq \phi(I),\eeqs and naturally there exists an epimorphism
\beqs \frac{\phi(\gim(M_n))}{\phi(\gim(M_n))\bigcap J}\rightarrow \gim(M_n)/I.\eeqs
\end{lemma}

For convenience, we let $\theta^*(t)=t^r\eta(t+t^{-1})$, where $r=\deg\eta$, then $\theta^*$  is a polynomial with $\theta^*(0)\not=0$ and it holds that
\beqs\hat{\sg}/J=\hat{\sg}/\theta^*(t).\eeqs
Moreover, $\theta^*(t)$ and $\eta(t+t^{-1})$ have same roots.

Suppose that $\hat{\sg}/\theta^*(t)$ has a Levi decomposition
\beqs \hat{\sg}/\theta^*(t)=S+U,\eeqs
where $S$ is semi-simple and $U$ is the solvable radical. Let ${\bf p}: S+U\rightarrow (S+U)/U$ be the canonical map.
If ${\bf i}:\gim(M_n)/I\rightarrow\hat{\sg}/\theta^*(t)\subseteq S$ is the injective map, then ${\bf p}\circ{\bf i}$ is injective. Hence we may assume that $$\theta^*(t)=\prod_{i=1}^K(t-a_i),$$
where $a_i\not=a_j$ if $i\not=j$.

Out of questions, $\gim(M_n)/I$ is a subalgebra of $\hat{\sg}/\theta^*(t)$ through isomorphism $\phi$.

We may assume that $c_i, d_i\in\mc^\times$ such that
\beqs \sum_{i=1}^Kc_i\frac{\theta^*(t)}{t-a_i}=1, \quad \frac{\theta^*(t)}{t-a_i}\equiv d_i (\hbox{\rm mod\;} t-a_i),\eeqs
we also assume that $\hat{\sg}/\theta^*(t)=\bigoplus_{i=1}^K \sg_i$, and $\psi^\clubsuit$ is the canonical map from $\hat{\sg}$ to $\hat{\sg}/\theta^*(t)$.

Let  $$\psi_{a_i}:\hat{\sg}\rightarrow\hat{\sg}_i=\frac{\sg\otimes \theta^*(t)/(t-a_i)\mc[t,t^{-1}]}{(t-a_i)}\cong\hat{\sg}/(t-a_i)$$ be the evaluation map, where $\psi_{a_i}(x\otimes t^m\frac{\theta^*(t)}{(t-a_i)})=d_ia_i^mx^{[i]}$.

\begin{lemma}\label{L6.4}
For any $x\otimes t^m$, we have
$$\psi^\clubsuit(x\otimes t^m)=\sum_{i=1}^Kc_id_i\psi_{a_i}(x\otimes t^m)=\sum_{i=1}^Kc_id_ia_i^mx^{[i]}.$$
Particularly, $c_id_i=1$ for any $1\leq i\leq K$, then $\psi^\clubsuit=\psi_{\underline{a}}$, where $\underline{a}=(a_1,\cdots,a_K)$.
\end{lemma}
\begin{proof}
Since $\sum_{i=1}^Kc_i\frac{\theta^*(t)}{t-a_i}=1$, evaluate $t=a_i$ in both sides, we have $c_id_i=1$.
For any $x\otimes t^m$, we have
\beqs \psi^\clubsuit(x\otimes t^m)&=&\sum_{i=1}^Kc_ix\otimes t^m\frac{\theta^*(t)}{t-a_i}\\
&=&\sum_{i=1}^Kc_id_i a_i^mx^{[i]}=\sum_{i=1}^K a_i^mx^{[i]}.\eeqs
\end{proof}

\begin{lemma}\label{L6.5}
Let $a\not=0, \pm1$, then associated to  the map
$$\psi_{a}\oplus\psi_{a^{-1}}:\gim(M_n)/I\rightarrow \hat{\sg}/(t-a)\oplus \hat{\sg}/(t-a^{-1}),$$ we have that
$$Im(\psi_{a}\oplus\psi_{a^{-1}})\cong Im(\psi_a).$$
\end{lemma}
\begin{proof}
As we know that
\beqs \sigma: &&\sg\otimes \mc[t, t^{-1}]\rightarrow \sg\otimes \mc[t, t^{-1}]\\
&&x\otimes t^m\mapsto x\otimes t^{-m},\eeqs
is an isomorphism. Moreover,
$\psi_a\circ \sigma=\psi_{a^{-1}}$, hence we have that
\beqs \psi_{a}\oplus\psi_{a^{-1}}=\psi_a\circ(1\oplus\sigma),\eeqs
notice that $1\oplus\sigma$ is  diagonally injective, then the image of  $\psi_{a}+\psi_{a^{-1}}$ is isomorphic to $Im(\psi_a)$.
\end{proof}

{\noindent\bf Proof for Theorem 2.3:}\quad
Undoubtedly, since $\eta$ is a polynomial  in  variable $t+t^{-1}$ and $\theta^*(0)\not=0$,  we have that $(t-a)|\theta^*(t)$ if and only if $(t-a^{-1})|\theta^*(t)$.   By Lemma \ref{L6.4}, $\psi^\clubsuit=\psi_{\underline{a}}$. Here, $a_k\in {\underline{a}}$ if and only if $a_k^{-1}\in {\underline{a}}$. However, the approach of proof for Lemma \ref{L5.1} still play a major role in this case. Follow the process of the proof for Lemma \ref{L5.1} and Lemma \ref{L5.3}, we can successfully  reach that:
\beqs x_i^{[k]}, y_i^{[k]}\in Im(\psi_{\underline{a}}),&& a_k=\pm1, 1\leq i\leq n\\
x_i^{[k]}+x_i^{[j]}, y_i^{[k]}+y_i^{[j]}\in Im(\psi_{\underline{a}}),&& a_k\not=\pm1, a_ka_j=1, 1\leq i\leq n.\eeqs

By Proposition \ref{P4.2} and Lemma \ref{L6.5}, the subalgebra generated by $x_i^{[k]}, y_i^{[k]} (a_k=\pm1, 1\leq i\leq n)$ is one of $so_{2n}, sp_{2n}$, and the subalgebra generated by $x_i^{[k]}+x_i^{[j]}, y_i^{[k]}+y_i^{[j]} (a_ka_j=1, 1\leq i\leq n)$ is $sl_{2n}$. Since $t-1$ or $t+1$ appears in $\theta(t)$ at most one time, each of $so_{2n}$ and $sp_{2n}$ appears at most one time. Then $\gim(M_n)/I$ is of type $M(n, {\bf a}, {\bf c}, {\bf d})$.

The proof is completed.

\vskip.5cm

{\bf Acknowledgements. }

The first author is partially supported by NSERC of Canada. The second author is supported by the  NNSF of China (Grant No. 11001110, 11271131) and {\it Jiangsu Government Scholarship for Overseas Studies}.

\end{document}